\documentclass[10pt]{amsart}

\newtheorem{theoreme}{Theorem}[section]

\newtheorem{lemme}[theoreme]{Lemma}

\newtheorem{proposition}[theoreme]{Proposition}

\newtheorem{definition}[theoreme]{Definition}

\newtheorem{remarque}[theoreme]{Remark}

\newtheorem{corollaire}[theoreme]{Corollary}

\input xy

\xyoption{all}

\usepackage{amscd}

\usepackage{amssymb}

\usepackage[a4paper, left=2cm,right=2cm,bottom=2cm,top=3.5cm]{geometry}
\usepackage{amsthm,array,amssymb,amscd,amsfonts,latexsym, url}
\usepackage{amsmath}
\usepackage{graphicx,epsfig}

\title{Remarks on approximate decompositions of the diagonal}
\author{Ren\'e Mboro}
\date{}
\begin{document}

\begin{abstract} In this paper, we investigate, for varieties over $\mathbb C$ with trivial group of $0$-cycles, the gap between essential $\mathrm{CH}_0$-dimension $2$ and essential $\mathrm{CH}_0$-dimension $0$. In particular, we present sufficient (and necessary) conditions for a variety with trivial group of $0$-cycles and essential $\mathrm{CH}_0$-dimension $\leq 2$ to have, in fact, essential $\mathrm{CH}_0$-dimension $0$.
\end{abstract}
\maketitle
\section{Introduction} The characterization of the complex smooth projective varieties of dimension $n\geq 3$ which are rational, i.e. birationally equivalent to $\mathbb P^n_{\mathbb C}$, is a long standing problem in algebraic geometry. If such a variety $X$ is rational then, by birational invariance of the group $\mathrm{CH}_0$ of $0$-cycles modulo rational equivalence, we have an isomorphism $\mathrm{CH}_0(X)=\mathbb Zx$ for a (any) point $x\in X(\mathbb C)$.\\
\indent However a more general class of varieties satisfy the isomorphism $\mathrm{CH}_0(X)\simeq \mathbb Z$. Among them, we can mention the class of rationally connected varieties, for which any pair of point is contained in a rational curve,  or the (conjecturally) smaller class of unirational varieties, i.e. varieties admitting a dominant rational map from a projective space, which, in dimension $\leq 2$ coincide with the class of rational varieties. In dimension $\geq 3$, efficient obstructions to rationality for unirational varieties have been found since the early $1970's$: Clemens and Griffiths (\cite{Cl-Gr}) proved that the intermediate Jacobian of a cubic threefold $X$, which is a unirational variety, is not isomorphic to a sum of Jacobian of curves which would be the case should $X$ be rational; in a $1989$ paper (\cite{CT-O}), Colliot-Th\'el\`ene and Ojanguren studied higher degree analog of the obstruction that Artin and Mumford used to provide an example of a unirational non stably rational threefold, namely the unramified cohomology groups, which are birational invariants trivial for stably rational varieties, and used them to prove non stable rationality of some unirational quadric bundles.\\
\indent In fact, the birational invariance of the $\mathrm{CH}_0$ group yields more for a rational variety $X$; for any field extension $L/\mathbb C$, we have, as for the projective space, $\mathrm{CH}_0(X)= \mathbb Zx_L$ where $x\in X(\mathbb C)$ and $x_L=x\times_{\mathbb C}Spec(L)$. The varieties satisfying this property are said to have universally trivial $\mathrm{CH}_0$ group or have their $0$-cycles universally supported on a point. We can reformulate this property, after (\cite{A-CT-P}), saying that there exists a Chow-theoretic decomposition of the diagonal i.e. than we can write the diagonal of $X$ as :
\begin{equation}\Delta_X = X\times x + Z\ \mathrm{in\ CH}^n(X\times X) 
\end{equation} 
where $x\in X(\mathbb C)$ and $Z$ is a cycle supported on $D\times X$ for some proper closed algebraic subset $D$ of $X$. In \cite{Vois_unirat}, Voisin proved the existence of a Chow theoretic decomposition of the diagonal to be a strictly subtler birational invariant than Clemens-Griffiths and the non triviality of unramified cohomology groups criteria by exhibiting examples of unirational varieties whose intermediate Jacobian is isomorphic to a Jacobian of curve (Clemens-Griffiths criterion), having trivial unramified cohomology groups and which do not admit a Chow theoretic decomposition of the diagonal or equivalently, whose $0$-cycles are not universally supported on a point.\\
\indent For a variety $X$ satisfying $\mathrm{CH}_0(X)=\mathbb Z$, we can also consider a weaker property than having the $0$-cycles universally supported on a point: let us say that the $\mathrm{CH}_0$ is universally supported on a subvariety $Y\subset X$ if the natural morphism $\mathrm{CH}_0(Y)\rightarrow \mathrm{CH}_0(X)$ is universally surjective, i.e. for any $L/\mathbb C$, $\mathrm{CH}_0(Y_L)\rightarrow \mathrm{CH}_0(X_L)$ is surjective. 
\begin{remarque}\label{rmk_intro} \normalfont An equivalent formulation is the existence of a decomposition of the diagonal
\begin{equation}\label{eq_decomp_approx_princ} \Delta_X=\Gamma_1 + \Gamma_2\ \mathrm{in\ CH}^n(X\times X)
\end{equation}
where $\Gamma_1$ is supported on $D\times X$ for some proper closed subset $D\subset X$ and $\Gamma_2$ is supported $X\times Y$. Indeed, applying universal surjectivity to the field $L=\mathbb C(X)$, we see that the diagonal point $\delta_X\in \mathrm{CH}_0(X_{\mathbb C(X)})$ can be written $i_{L,*}z$ in $\mathrm{CH}_0(X_{\mathbb C(X)})\simeq \mathrm{CH}^n(X_{\mathbb C(X)})$, where $z\in \mathrm{CH}_0(Y_L)$ and $i:Y\hookrightarrow X$. This implies (\ref{eq_decomp_approx_princ}), using Bloch identity $\mathrm{CH}^n(X_{\mathbb C(X)})= \varinjlim_{U\subset X,\ open} \mathrm{CH}^n(U\times X)$ and the localization exact sequence.
\end{remarque}
Voisin made in \cite{Vois_cub} the following definition:
\begin{definition}\textit{(\cite[Definition 1.2]{Vois_cub})}The essential $\mathrm{CH}_0$-dimension of a variety is the minimal integer $k$ such that there is a closed subscheme $Y\subset X$ of dimension $k$ such that the $\mathrm{CH}_0$ group of $X$ is universally supported on $Y$.
\end{definition}
In the case of cubic hypersurfaces, the following is proved in \cite{Vois_cub}:
\begin{theoreme}\textit{(\cite[Theorem 1.4]{Vois_cub})} The essential $\mathrm{CH}_0$-dimension of a very general $n$-dimensional cubic hypersurface over $\mathbb C$ is either $n$ or $0$, for $n=4$ or $n$ odd.
\end{theoreme}  

In \cite{CT_cub}, Colliot-Th\'el\`ene proved the following proposition concerning varieties of essential $\mathrm{CH}_0$-dimension $\leq 1$:

\begin{theoreme}\textit{(\cite[Proposition 2.5]{CT_cub})}\label{thm_CT_1} Let $X$ a smooth projective variety over $\mathbb C$ satisfying $\mathrm{CH}_0(X)=\mathbb Z$. Assume $X$ has essential $\mathrm{CH}_0$-dimension $\leq 1$. Then the essential $\mathrm{CH}_0$-dimension of $X$ is $0$ (i.e. $X$ admits a Chow-theoretic decomposition of the diagonal).
\end{theoreme}
\indent The goal of this paper is to investigate how far a variety with trivial $\mathrm{CH}_0$, which has essential $\mathrm{CH}_0$ dimension $\leq 2$ is from having a universally trivial $\mathrm{CH}_0$ group i.e. essential $\mathrm{CH}_0$ dimension $0$.
A first result in this direction is the following:
\begin{theoreme}\textit{(\cite[Corollary 2.2]{Vois_cub}, \cite[Proposition 1.9]{A-CT-P})} Let $\Sigma$ be a smooth complex projective surface. Assume $\mathrm{CH}_0(\Sigma)=\mathbb Z$ and $\mathrm{Tors}(H^*(\Sigma,\mathbb Z))=0$. Then $\Sigma$ has $\mathrm{CH}_0$ universally trivial.
\end{theoreme} 
We show the following generalizations:
\begin{theoreme}\label{thm_main1} Let $X$ be a smooth complex projective variety of dimension $n$ such that $\mathrm{CH}_0(X)=\mathbb Z$. Assume $X$ has essential $\mathrm{CH}_0$ dimension $\leq 2$. If
\begin{enumerate}
\item\label{cond1_thm1} $\mathrm{Tors}(H^2(X,\mathbb Z))=0$ and
\item\label{cond2_thm1} $H^3(X,\mathbb Z)=0$,
\end{enumerate}
then $X$ has universally trivial $\mathrm{CH}_0$ group i.e. has essential $\mathrm{CH}_0$ dimension $0$.
\end{theoreme}

In another direction, we have the following result. Let us introduce first the following condition:\\

\indent \emph{(*) there is a smooth projective variety }$\widetilde{Y}$\emph{ of dimension }$(dim(X)-1)$\emph{ and a morphism }$j:\widetilde{Y}\rightarrow X$ \emph{such that} $$j_*:\mathrm{Pic}^0(\widetilde{Y})\rightarrow \mathrm{CH}^2(X)_{alg}$$\emph{ is universally surjective.}\\

\indent 
In more geometric terms, the condition means that any family of algebraically trivial codimension $2$ cycles factors generically through $j_*$. Indeed, let $\mathcal{Z}\in \mathrm{CH}^2(B\times X)$ be such a family, parametrized by a smooth projective base $B$. Applying condition (*) to the field $\mathbb C(B)$, gives that the cycle $\mathcal{Z}_{\mathbb C(B)}\in \mathrm{CH}^2(X_{\mathbb C(B)})_{alg}= \varinjlim_{U\subset B,\ open}\mathrm{CH}^2(U\times X)_{alg}$ has a pre-image $D\in \mathrm{Pic}^0(Y)(\mathbb C(B))$ by $j_{\mathbb C(B),*}$. The $\mathbb C(B)$-point $D$ corresponds naturally to a rational map (thus a morphism since $\mathrm{Pic}^0(Y)$ is an abelian variety) $\mathcal D:B\rightarrow \mathrm{Pic}^0(Y)$. The identity $j_{\mathbb C(B),*}(D)=\mathcal{Z}_{\mathbb C(B)}$ says that the applications $B(\mathbb C)\rightarrow \mathrm{CH}^2(X)_{alg}$, given by $b\mapsto \mathcal{Z}_b$ and $B(\mathbb C)\stackrel{\mathcal D}{\rightarrow}\mathrm{Pic}^0(Y)(\mathbb C)\stackrel{j_*}{\rightarrow}\mathrm{CH}^2(X)_{alg}\simeq J^3(X)$ coincide on a dense open set of $B$ hence everywhere since targets are abelian varieties.\\
\indent Conversely, any cycle $\mathcal{Z}_L\in \mathrm{CH}^2(X_L)_{alg}$ has a model $\mathcal Z$ which is a family of algebraically trivial codimension $2$ cycles of $X$ parametrized by a smooth quasi-projective model $B$ of $L$. The factorization of that family through $j_*$ gives rise to a morphism $\mathcal D:B\rightarrow \mathrm{Pic}^0(Y)$ which, passing to the limit over the open sets of $B$ yields a $\mathbb C(B)=L$ point of $\mathrm{Pic}^0(Y)$ mapped by $j_{L,*}$ to $\mathcal{Z}_L$.\\
\indent Now, let us state the second theorem of the paper:

\begin{theoreme}\label{thm_main2} Let $X$ be a smooth complex projective variety of dimension $n$ such that $\mathrm{CH}_0(X)=\mathbb Z$. Assume $X$ has essential $\mathrm{CH}_0$ dimension $\leq 2$. If $X$ satisfies the condition (*), then $X$ has universally trivial $\mathrm{CH}_0$ group i.e. has essential $\mathrm{CH}_0$ dimension $0$.
\end{theoreme}

\begin{remarque} \normalfont We observe, conversely, that if $X$ has universally trivial $\mathrm{CH}_0$ group, then the conditions (\ref{cond1_thm1}) of Theorem \ref{thm_main1} and (*) are satisfied (see Lemma \ref{lem_recip_main_thm}).
\end{remarque}
The key condition $(*)$ appearing in the theorem is expressed in terms of universal generation, a notion introduced by Shen in \cite{Shen_univ}, where he uses universal generation of $1$-cycles on cubic hypersurfaces (of dimension $\geq 3$) to relate the existence of a decomposition of the diagonal for cubic $3$-folds and $4$-folds to the algebraicity of some cohomological classes on their Fano varieties of lines associated to the pairing in the middle cohomology of cubic hypersurfaces.\\
\indent In the case of threefolds, another relation between essential $\mathrm{CH}_0$-dimension and condition (*) is presented in Section \ref{applications}. Combined with Theorem \ref{thm_main2}, it yields the following result:
\begin{theoreme}\label{thm_intro_csq} The essential $\mathrm{CH}_0$-dimension of a very general Fano complete intersection threefold is $0$ or $3$. In particular, the essential $\mathrm{CH}_0$-dimension of the very general quartic threefold is equal to $3$.
\end{theoreme}
\indent The paper is organized as follows: the first section is devoted to the proofs of Theorems \ref{thm_main1} and \ref{thm_main2}. The second section is devoted to the analysis of condition $(*)$; we try to relate it, at least in some special case to a more geometric condition. The third second is devoted to the proof of Theorem \ref{thm_intro_csq}.

\section{Main theorems}
Let us begin this section by a lemma which proves that the conditions (\ref{cond1_thm1}) of Theorem \ref{thm_main1} and (*) are also necessary.
\begin{lemme}\label{lem_recip_main_thm} Let $X$ be a smooth projective variety of dimension $n$ whose essential $\mathrm{CH}_0$ dimension is $0$. Then 
\begin{enumerate} \item\label{cond1_lem1} $\mathrm{Tors}(H^2(X,\mathbb Z))=0=\mathrm{Tors}(H^3(X,\mathbb Z))$ and 
\item\label{cond2_lem1} $X$ satisfies $(*)$.
\end{enumerate}
\end{lemme}
\begin{proof} Saying that $X$ has essential $\mathrm{CH}_0$-dimension $0$ is equivalent to the existence of a Chow-theoretic decomposition of the diagonal of $X$. We have: \begin{equation}\label{eq_decomp_diag_lemme}\Delta_X =X\times x + Z\ \mathrm{in\ CH}^n(X\times X)\end{equation} where $x\in X(\mathbb C)$ and $Z$ is a cycle supported on $D\times X$ for some proper closed algebraic subset $D$ of $X$. We can choose $D$ such that, denoting $\widetilde{D}$ a desingularization of $D$ and $j:\widetilde{D}\rightarrow X$ the composition of the desingularization followed by the inclusion, the cycle $Z$ lifts to a cycle $\widetilde{Z}\in \mathrm{CH}^{n-1}(\widetilde{D}\times X)$. Item (\ref{cond1_lem1}) is proved in \cite[Theorem 4.4]{Vois_abel}. Let us prove (\ref{cond2_lem1}). Let $L/\mathbb C$ be a field extension and $\gamma\in \mathrm{CH}^2(X_L)_{alg}$. Letting both sides of the extension of (\ref{eq_decomp_diag_lemme}) to $L$ act on $\gamma$, we get the equality: $$\gamma=\Delta_{X_L}^*\gamma= j_{L,*}(\widetilde{Z}_L^*\gamma)\ \mathrm{in\ CH}^2(X_L)$$ with $\widetilde{Z}_L^*\gamma\in \mathrm{CH}^1(\widetilde{D}_L)_{alg}=\mathrm{Pic}^0(\widetilde{D}_L)=\mathrm{Pic}^0(\widetilde{D})_L$, which proves the universal surjectivity of $j_*$.
\end{proof}

\begin{proof}[Proof of Theorem \ref{thm_main1}.] Let us assume that $X$ satisfies conditions (\ref{cond1_thm1}) and (\ref{cond2_thm1}) of the Theorem \ref{thm_main1}. By Remark \ref{rmk_intro}, 
the diagonal of $X$ can be written:
\begin{equation}\label{eq_decomp_approx} \Delta_X=\Gamma_1 + \Gamma_2\ \mathrm{in\ CH}^n(X\times X)
\end{equation} 
where $\Gamma_1$ is supported on $D\times X$ for some proper closed subset $D\subset X$ and $\Gamma_2$ is supported $X\times \Sigma$. Let $\tau:\widetilde{\Sigma}\rightarrow \Sigma$ be a desingularization of $\Sigma$. Enlarging $\Sigma$ if necessary we can find $\widetilde{\Gamma}_2\in \mathrm{CH}^{2}(X\times \widetilde{\Sigma})$ such that $(id_X, i\circ\tau)_*\widetilde{\Gamma}_2=\Gamma_2$ in $\mathrm{CH}^n(X\times X)$, where $i:\Sigma\hookrightarrow X$ is the inclusion. To get a Chow-theoretic decomposition of the diagonal, it is sufficient to prove that $\widetilde{\Gamma}_2$ (hence $\Gamma_2$ in $\mathrm{CH}^n(X\times X)$) can be decomposed as $X\times x + Z$ in $\mathrm{CH}^2(X\times\widetilde{\Sigma})$ for a cycle $Z$ supported on $D'\times \widetilde{\Sigma}$, $D'$ being a proper closed subset of $X$. We have the following proposition:
\begin{proposition}\textit{(\cite[Proposition 2.1]{Vois_cub})} Let $Y$ be a smooth projective variety. If $Y$ admits a decomposition of the diagonal modulo algebraic equivalence, that is $$\Delta_Y = Y\times y +Z\ \mathrm{in\ CH}^{dim(Y)}(Y\times Y)/alg$$ with $Z$ supported on $D\times Y$ for some proper closed algebraic subset $D$ of $Y$, then $Y$ admits a Chow-theoretic decomposition of the diagonal.
\end{proposition}
We include the proof for the sake of completeness:
\begin{proof}[Sketch of proof] The result is obtained as a consequence of a nilpotence result of Voevodsky (\cite{Voe}) and Voisin (\cite{Vois_nilp}) asserting that given a self-correspondence $\Gamma\in \mathrm{CH}^n(Y\times Y)$ that is algebraically trivial, there is an integer $N$ such that $\Gamma^{\circ N}=0$ in $\mathrm{CH}^n(Y\times Y)$.\\
\indent Applying the nilpotence result to the algebraically trivial self-correpondence $(\Delta_Y-(Y\times y) + Z)$ yields the result since, writting down the different terms, using the fact that $Z\circ (Y\times y)=0$ and $(\Delta_Y-Y\times y)\circ (\Delta_Y-Y\times y)=\Delta_Y-Y\times y$, we see that any power of $(\Delta_Y-(Y\times y) + Z)$ is of the form $\Delta_Y-(Y\times y) + Z'$ for a cycle $Z'$ supported on $D\times Y$.
\end{proof}
\textit{}\\
We conclude from this proposition that in order to get the decomposition of the diagonal of $X$, it suffices to decompose $\widetilde{\Gamma}_2$  as $X\times x + Z$ for a cycle $Z$ supported on $D'\times \widetilde{\Sigma}$, $D'$ being a proper closed subset of $X$, in $\mathrm{CH}^2(X\times\widetilde{\Sigma})/alg$. Indeed, this will decompose $\Gamma_2$ as $X\times x+Z$, with $Z$ supported on $D\times X$ in $\mathrm{CH}^n(X\times X)/alg$.\\
\indent Now, since $\mathrm{CH}_0(X)=\mathbb Z$, we have $\mathrm{CH}_0(X\times \widetilde{\Sigma})\simeq \mathrm{CH}_0(\widetilde{\Sigma})$ i.e. the group of $0$-cycles of $X\times\widetilde{\Sigma}$ supported on a $2$-dimensional subscheme of $X\times \widetilde{\Sigma}$. Then, by work of Bloch and Srinivas (\cite[Theorem 1 (ii)]{Bl-Sr}), algebraic and homological equivalences coincide on $\mathrm{CH}^2(X\times\widetilde\Sigma)$ so that, in order to get the equality $\Gamma_2=X\times x +Z$ in $\mathrm{CH}^2(X\times\widetilde{\Sigma})/alg$, it is sufficient to prove the corresponding cohomological decomposition, that is to prove that $[\widetilde{\Gamma}_2]$ can be written $[X\times x] + [Z]$ in $H^4(X\times\widetilde{\Sigma})$ for a cycle $Z$ supported on $D'\times \widetilde{\Sigma}$, $D'$ being a proper closed subset of $X$.\\
\indent We have the K\"unneth exact sequence: $$0\rightarrow \bigoplus_{i+j=4}H^i(X,\mathbb Z)\otimes H^j(\widetilde\Sigma,\mathbb Z)\rightarrow H^4(X\times\widetilde\Sigma,\mathbb Z)\rightarrow \bigoplus_{i+j=5}\mathrm{Tor}_1(H^i(X,\mathbb Z),H^j(\widetilde\Sigma,\mathbb Z))\rightarrow 0$$ from which, we see, using the fact that the groups $H^{*\leq 1}(*,\mathbb Z)$ are always torsion-free and the assumption $\mathrm{Tors}(H^2(X,\mathbb Z))=0=\mathrm{Tors}(H^3(X,\mathbb Z))$, that $H^4(X\times \widetilde{\Sigma},\mathbb Z)$ admits a K\"unneth decomposition.\\ 

Let us denote $\delta^{i,j}\in H^i(X,\mathbb Z)\otimes H^j(\widetilde\Sigma,\mathbb Z)$ the K\"unneth components of $[\widetilde{\Gamma}_2]$. They are Hodge classes since the projection on K\"unneth types are morphism of Hodge structures.\\
\indent The component $\delta^{3,1}$ is $0$ since by assumption (\ref{cond2_thm1}), $H^3(X,\mathbb Z)=0$.\\
\indent The component $\delta^{0,4}$ is of the form $[X\times z]$ for a $0$-cycle $z$ on $\widetilde{\Sigma}$.\\
\indent Let us write $\widetilde{\Sigma}=\sqcup_i \widetilde{\Sigma}_i$ where the $\widetilde{\Sigma}_i$ are smooth connected surfaces. Since $H^0(\widetilde{\Sigma},\mathbb Z)=\oplus_i\mathbb Z[\widetilde{\Sigma}_i]$, for each $i$, the component $\delta_i^{4,0}\in H^4(X,\mathbb Z)\otimes H^0(\widetilde{\Sigma}_i,\mathbb Z)$ can be written $pr_1^*\alpha_i$ for a cohomology class $\alpha_i\in H^4(X,\mathbb Z)$ and by projection formula $pr_{1,*}(\delta_i^{4,0}\cup [X\times x_i])= pr_{1,*}(pr_1^*\alpha_i\cup [X\times x_i])=\alpha_i$ for a (any) point $x_i\in\widetilde{\Sigma}_i$. Now, we have $[\widetilde{\Gamma}_2']\cup[X\times x_i]=[\widetilde{\Gamma}_2'\cdot (X\times x_i)]=\delta_i^{4,0}\cup [X\times x_i]$; applying $pr_{1,*}$ to these equalities yields $[pr_{1,*}(\widetilde{\Gamma}_2'\cdot(X\times x_i))]=\alpha_i$ i.e. $\delta_i^{4,0}=pr_1^*\alpha_i$ is algebraic and supported on $pr_{1,*}(\widetilde{\Gamma}_2'\cdot(X\times x_i))\times \widetilde{\Sigma}_i$ and $pr_{1,*}(\widetilde{\Gamma}_2'\cdot(X\times x_i))$ has codimension $2$ in $X$, in particular it does not dominate $X$. So the component $\delta^{4,0}=\sum_i\delta_i^{4,0}$ is algebraic and represented by a cycle which does not dominate $X$ by the first projection.\\
\indent As $\mathrm{CH}_0(X)\simeq\mathbb Z$, by \cite[Proposition 1]{Bl-Sr}, there is an integer $N\neq 0$ such that 
\begin{equation}\label{eq_decomp_rat}
N\Delta_X=N(X\times x)+Z\ \mathrm{in\ CH}^n(X\times X)
\end{equation} where $Z$ is supported on $D'\times X$ for some proper closed subset $D'$ of $X$. Looking at the action on $H^1(X,\mathbb Z)$, we see that $H^1(X,\mathbb Z)$ is a torsion group annihilated by $N$ but since $H^1(X,\mathbb Z)$ is torsion-free, $H^1(X,\mathbb Z)=0$. Hence, $\delta^{1,3}=0$. Letting the correspondences of (\ref{eq_decomp_rat}) act on the complex vector space $H^0(X,\Omega_X^2)$ we see that it is annihilated by $N$ i.e. it is $0$, so that by Lefschetz theorem on $(1,1)$ classes, $H^2(X,\mathbb Z)$ is algebraic. So the Hodge class $\delta^{2,2}$ belongs to $H^{1,1}(X)\otimes H^{1,1}(\widetilde{\Sigma})$; it is thus algebraic and of the form $[\sum_iD_i\times C_i]$ where the $D_i$ are divisors on $X$ and the $C_i$ are curves on $\widetilde{\Sigma}$, in particular it does not dominate $X$ by the first projection.
\end{proof}

\begin{proof}[Proof of Theorem \ref{thm_main2}]Let us assume that $X$ satisfies condition (*) and has essential $\mathrm{CH}_0$-dimension $\leq 2$. We thus have:
\begin{equation}\label{eq_decomp_approx'} \Delta_X=\Gamma_1 + \Gamma_2\ \mathrm{in\ CH}^n(X\times X)
\end{equation} 
where $\Gamma_1$ is supported on $D\times X$ for some proper closed subset $D\subset X$ and $\Gamma_2$ is supported $X\times \Sigma$.\\
\indent Let us write $\widetilde{\Sigma}=\sqcup_i \widetilde{\Sigma}_i$ where the $\widetilde{\Sigma}_i$ are smooth connected surfaces. Choose a point $\sigma_i\in \widetilde{\Sigma}_i$. For each $i$, the cycle $\widetilde{\Gamma}_{2,i}:=\widetilde{\Gamma}_{2|\widetilde{\Sigma}_i}\in \mathrm{CH}^2(X\times \widetilde{\Sigma}_i)$ can be written as $Z_i\times\widetilde{\Sigma}_i+ \widetilde{\Gamma}_{2,i,alg}$ where $Z_i\in \mathrm{CH}^2(X)$ is defined as $\widetilde{\Gamma}_{2,i,*}(\sigma_i)$ and $\widetilde{\Gamma}_{2,i,alg}$ is a family of cycles algebraically equivalent to $0$ on $X$ parametrized by $\widetilde{\Sigma}_i$. By condition (*) applied to each field $\mathbb C(\widetilde{\Sigma}_i)$, we get a cycle $\mathcal Z_i\in Pic(Y\times \widetilde{\Sigma}_i)$ such that $\mathcal Z_i - \widetilde{\Gamma}_{2,i,alg}$ vanishes in $\mathrm{CH}^2(X\times U_i)$ where $U_i\subset \widetilde{\Sigma}_i$ is a dense open subset. By the localization exact sequence, we conclude that the cycle $\mathcal Z_i - \widetilde{\Gamma}_{2,i,alg}$ is supported on $\cup_j X\times C_{i,j}$ where $\cup_j C_{i,j}=\widetilde{\Sigma}_i\backslash U_i$. Putting everything together, we conclude that $\Delta_X = Z' +Z''$ where $Z'$ is supported  on $D'\times X$ and $Z''$ is supported on $X\times C$ where $C=\cup_{i,j} C_{i,j}$. We thus conclude that the essential $\mathrm{CH}_0$-dimension of $X$ is $\leq 1$ and the proof is concluded by applying Theorem \ref{thm_CT_1}.
\end{proof}
\section{Universal generation of codimension $2$ cycles}
In this section, we discuss the relation of condition (*) to the existence of a universal codimension $2$ cycle.\\

\indent Let $X$ be a smooth projective complex variety satisfying $\mathrm{CH}_0(X)=\mathbb Z$. Then by a theorem of Roitman $H^{i,0}(X)=0$ for any $i>0$, so that the Hodge structure on $H^3(X,\mathbb Z)$ has level $1$ and $H^{2n-1}(X,\mathbb Q)=0$. So $H^3(X,\mathbb Z)_{prim}:= Ker(c_1(\mathcal O_X(1))^{n-3+1}\cup:H^3(X,\mathbb Z)_{/Tors}\rightarrow H^{2n-1}(X,\mathbb Z)_{/Tors})$ is the whole of $H^3(X,\mathbb Z)_{/Tors}$ so that the bilinear form defined on $H^3(X,\mathbb Z)_{/Tors}$ using a polarization $\mathcal O_X(1)$, polarizes the intermediate Jacobian for codimension $2$ cycles $J^3(X)$ is an abelian variety. By work of Bloch and Srinivas, we have in our setting $\mathrm{CH}^2(X)_{alg}= \mathrm{CH}^2(X)_{hom}\simeq J^3(X)(\mathbb C)$. The condition $(*)$ is related to codimension $2$-cycles. In \cite[Theorem 2.1]{Vois_unirat}, Voisin exhibited a (birationally invariant) necessary condition for stable rationality, namely the existence of a universal codimension $2$ cycle i.e. the existence of a correspondence $\mathcal Z\in \mathrm{CH}^2(J^3(X)\times X)$ such that the induced Abel-Jacobi morphism $\Phi_{\mathcal Z}:J^3(X)\rightarrow J^3(X)$, given by $t\mapsto \rho(\mathcal Z_t-\mathcal Z_{t_0})$, where $\rho:\mathrm{CH}^2(X)_{alg}\rightarrow J^3(X)$ is the natural regular morphism in the sense of Murre (\cite{Murre_cd2}), is the identity. We have the following relation with condition (*):
\begin{proposition}\label{prop_scinde1} Let $X$ be a smooth projective complex variety satisfying $\mathrm{CH}_0(X)=\mathbb Z$ and condition $(*)$. Assume moreover that $j_*:\mathrm{Pic}^0(\widetilde{Y})\rightarrow J^3(X)\simeq \mathrm{CH}^2(X)_{alg}$ is split. Then there is a universal codimension $2$ cycle $\mathcal{Z}\in \mathrm{CH}^2(J^3(X)\times X)$.
\end{proposition}
\begin{proof} As $j_*$ is split, the splitting morphism gives an imbedding $s:J^3(X)\hookrightarrow \mathrm{Pic}^0(\widetilde{Y})$. Denoting $\mathcal P$ the Poincar\'e divisor on $\mathrm{Pic}^0(\widetilde{Y})\times \widetilde{Y}$, set $\mathcal Z= (id_{J^3(X)},j)_*(s,id_{\widetilde{Y}})^*\mathcal P$ in $\mathrm{CH}^2(J^3(X)\times X)$. Then, by construction, $\mathcal Z$ is a universal codimension $2$ cycle. Indeed, the Abel-Jacobi morphism $\Phi_{\mathcal Z}:J^3(X)\rightarrow J^3(X)$ is just given by $j_*\circ s_*$ which is the identity by definition of $s$. 
\end{proof}

We have this other proposition relating condition (*) to the existence of a universal codimension $2$ cycle:
\begin{proposition}\label{prop_facteur_jacobienne} Let $X$ be a smooth projective complex variety satisfying $\mathrm{CH}_0(X)=\mathbb Z$. Assume $J^3(X)$ is a direct factor of a sum of Jacobian of curves $\oplus_i J(C_i)$. Then the existence a universal codimension $2$ cycle $\mathcal{Z}\in \mathrm{CH}^2(J^3(X)\times X)$ implies condition $(*)$.
\end{proposition}
\begin{proof} Let us denote $p:\oplus_i J(C_i)\rightarrow J^3(X)$ the projection to $J^3(X)$ and $s:J^3(X)\rightarrow \oplus_i J(C_i)$ a section. Using the morphisms 
$$C_i\stackrel{j_i}{\hookrightarrow}J(C_i)\stackrel{p}{\rightarrow}J^3(X)$$
we get a correspondence $\mathcal{Y}\in \mathrm{CH}^2(\sqcup_iC_i\times X)$, defined on each $C_i\times X$ as $(p\circ j_i, id_X)^*\mathcal Z$, such that the induced Abel-jacobi morphism $\Phi_{\mathcal Y}:\sqcup_iC_i\rightarrow J^3(X)$ coincides with $\oplus_i p\circ j_i$. Let us denote $D=pr_{2,*}(\mathcal Y)\in \mathrm{CH}^1(X)$ and $\widetilde{D}$ a desingularization of (a divisor in the class) $D$. We have a morphism $j:Pic^0(\widetilde{D})\rightarrow J^3(X)$ and Abel-Jacobi morphism $\Phi_{\mathcal Y}:\sqcup_iC_i\rightarrow J^3(X)$ naturally factors through the morphism $j$ so that, by the universal property of the Albanese variety, the morphism $p:\oplus_i J(C_i)\rightarrow J^3(X)$ also factors through $j$.\\
\indent Now, let $\mathcal K\in \mathrm{CH}^2(W\times X)$ be a family of codimension $2$ cycles of $X$ parametrized by a smooth quasi-projective base $W$. Let us consider the Abel-Jacobi morphism $\Phi_{\mathcal K}:W\rightarrow J^3(X)$; we have the equality $p\circ s\circ\Phi_{\mathcal K}=\Phi_{\mathcal K}$ but $p$ factors through $j:Pic^0(\widetilde{D})\rightarrow J^3(X)$.
\end{proof}
\begin{remarque} \normalfont Instead of the existence of a universal codimension $2$ cycle, we can consider weaker conditions. Let us introduce the following conditions:
\begin{enumerate}\item there exist a smooth projective variety $W$ and a cycle $\mathcal Z\in \mathrm{CH}^2(W\times X)$ such that the induced morphism $\mathcal Z_*:\mathrm{Alb}(W)\rightarrow J^3(X)$ is an isomorphism;
\item there is a universally generating cycle of codimension $2$ i.e. there exist a quasi-projective variety $W$ and a cycle $\mathcal Z\in\mathrm{CH}^2(W\times X)$ such that $$\mathcal Z_*:\mathrm{CH}_0(W)\rightarrow \mathrm{CH}^2(X)_{alg}$$ is universally surjective.
\end{enumerate} 
But the hypotheses we have to add to get condition (*) are not clear.
\end{remarque}

\section{Applications}\label{applications}
\subsection{The case of threefolds} As we have seen in Remark \ref{rmk_intro}, when $X$ has dimension $3$, the fact that the $\mathrm{CH}_0$ group of $X$ is universally supported on a surface $\Sigma$ can be written in $\mathrm{CH}^3(X\times X)$, as 
\begin{equation}\label{eq_decomp_approx2}
\Delta_X=\Gamma_1 + \Gamma_2
\end{equation} where $\Gamma_2$ is supported $X\times \Sigma$ and $\Gamma_1$ is supported on $D\times X$ for some proper closed subset $D\subset X$, in particular $D$ is a $2$-dimensional closed subscheme of $X$. We thus have a symmetric situation in term of dimension of the supports of the $\Gamma_i$. We recall that in this setting, $J^3(X)$ is a principally polarized abelian variety, whose polarization is given by the unimodular intersection form $\langle\cdot,\cdot\rangle_{J^3(X)}$ induced by the intersection form on $H^3(X,\mathbb Z)$ via the isomorphism $H_1(J^3(X),\mathbb Z)\simeq H^3(X,\mathbb Z)_{/Tors}$. We have the following result:
\begin{theoreme}\label{thm_symm3} Let $X$ be a smooth threefold satisfying $\mathrm{CH}_0(X)=\mathbb Z$ and whose essential $\mathrm{CH}_0$-dimension is $\leq 2$. Assume that any endomomorphism of $J^3(X)$ is self-adjoint for $\langle\cdot,\cdot\rangle_{J^3(X)}$. Then $X$ satisfies condition $(*)$. So that, by Theorem \ref{thm_main2}, the essential $\mathrm{CH}_0$-dimension of $X$ is $0$.
\end{theoreme} 
\begin{proof} Let us consider the decomposition (\ref{eq_decomp_approx2}). For any $z\in \mathrm{CH}^2(X)_{alg}\simeq J^3(X)(\mathbb C)$, letting both sides act on $z$, we get $$z=\Gamma_1^*(z) + \Gamma_2^*(z).$$ By construction, denoting $\widetilde{D}\rightarrow X$ a desingularization (followed by the inclusion) of $D$, $\Gamma_1^*$ factorizes through $j_*:\mathrm{Pic}^0(\widetilde{D})\rightarrow J^3(X)$. Moreover, as $\Gamma_2^*$ is an endomorphism of $J^3(X)$, it is self-adjoint so that $\Gamma_2^*(z)=\Gamma_{2,*}(z)$ and the last term factorizes through $\tilde{i}_*:\mathrm{Pic}^0(\widetilde{\Sigma})\rightarrow J^3(X)$ where $\tilde{i}:\widetilde{\Sigma}\rightarrow X$ is the desingularization (followed by the inclusion) of $\Sigma$. Thus, we see that $j_*+\tilde{i}_*:\mathrm{Pic}^0(\widetilde{D}\sqcup \widetilde{\Sigma})\rightarrow J^3(X)$ is surjective. Now, let $\mathcal Z$ be a family of algebraically trivial codimension $2$ cycles parametrized by a smooth quasi-projective variety $T$. Since the identity on $\mathrm{CH}^2(X)_{alg}$ factors as $\Gamma_{1}^*+\Gamma_{2,*}$, the map $T\rightarrow \mathrm{CH}^2(X)_{alg}$ factors through $\mathrm{Pic}^0(\widetilde{D}\sqcup \widetilde{\Sigma})$. So $X$ satisfies (*). We conclude applying Theorem \ref{thm_main2}.
\end{proof}

As an application of Theorem \ref{thm_symm3}, let us state the following result:
\begin{corollaire}\label{prop_quart} The essential $\mathrm{CH}_0$-dimension of a very general Fano complete intersection of dimension $3$ is either $0$ or $3$. In particular a very general quartic threefold has essential $\mathrm{CH}_0$-dimension equal to $3$.
\end{corollaire}
\begin{proof} Let us begin by the following lemma:
\begin{lemme}\label{lem_monodr_quart} For a very general complete intersection $X$ of dimension $3$, we have $End_{HS}(H^3(X,\mathbb Q))=\mathbb QId$ where $End_{HS}(H^3(X,\mathbb Q))$ denote the space of endomorphisms of Hodge structure of $H^3(X,\mathbb Q)$. 
\end{lemme}
\begin{proof} The proof works as in \cite[Lemma 5.1]{Vois_cub}. Indeed, by \cite{Beau_monodr}, the monodromy group of a smooth complete intersection of dimension $3$ is Zariski dense in the symplectic group of $H^3(X,\mathbb Q)=H^3(X,\mathbb Q)_{prim}$. By \cite{Vois_Hdg_loci}, the Mumford-Tate group of the Hodge structure on $H^3(X,\mathbb Q)$ contains a finite index sub-group of the monodromy group so that the Mumford-Tate group of $H^3(X,\mathbb Q)$ in the above cases is the symplectic group.
\end{proof}
Let $X$ be a very general Fano complete intersection of dimension $3$, then the endomorphisms of $J^3(X)$ are all self-adjoint for $\langle\cdot,\cdot\rangle_{J^3(X)}$. 
Suppose $X$ has essential $\mathrm{CH}_0$-dimension $< 3$. Then by Theorem \ref{thm_symm3}, $X$ has essential $\mathrm{CH}_0$-dimension $0$. So we have the alternative.\\
\indent For quartic threefolds, it was proved in (\cite{CTP_quartiques}) that a very general quartic threefold does not admit a Chow-theoretic decomposition of the diagonal, so its essential $\mathrm{CH}_0$-dimension is $3$.
\end{proof}
\section*{Acknowledgments}
I am grateful to my advisor Claire Voisin for having brought to me this interesting question and for her guidance.

\end{document}